\documentclass[A4paper,12pt]{article}
\usepackage[utf8]{inputenc}
\usepackage{tikz}
\usepackage{url}

\usepackage{amsmath,amsthm,amscd,amssymb,eucal,mathrsfs}
\usepackage{amsfonts}
\usepackage{latexsym}
\usepackage{graphicx} 
\usepackage{fullpage}
\usepackage{multicol}
\usepackage{dsfont}
\usepackage{natbib}
\usepackage[all]{xy}
\usepackage{multirow}
\usepackage{color}

\newtheorem{thm}{Theorem}[section]

\newtheorem{cor}[thm]{Corollary}

\newtheorem{definition}[thm]{Definition}

\newtheorem{Proposition}[thm]{Proposition}

\newtheorem*{Satz*}{Satz}

\newtheorem{Lemma}[thm]{Lemma}

\newcommand{\mathset}[1]{{\left\{#1\right\}}}
\newcommand{\absolute}[1]{\left\lvert#1\right\rvert}

\DeclareMathOperator{\etale}{\acute{e}t}
\DeclareMathOperator{\Gal}{Gal}
\DeclareMathOperator{\Spec}{Spec}

\DeclareMathOperator{\GL}{GL}

\DeclareMathOperator{\supp}{supp}

\DeclareMathOperator{\Hom}{Hom}

\title{Hearing Tamagawa Factors modulo $q-1$}
\author{Patrick Erik Bradley}
\date{\today}

\begin{document}

\maketitle

\begin{abstract}
The Tamagawa factor w.r.t.\ the prime $p$ of an abelian variety $A$ over a non-archimedean local field $K$ is found to be congruent modulo $q-1$ to the wavelet eigenvalues of a $p$-adic Laplacian integral operator on the $O_K$-rational points of its N\'eron model, where $q=p^f$ is the cardinality of the residue field of $K$. The method is to express the volume of the $K$-rational points of $A$ w.r.t.\ the canonical measure in terms of the local $L$-factor given by the Frobenius  action on $\ell$-adic cohomolgy, and the Tamagawa factor; and then observe that this coincides with the Serre invariant of that compact $p$-adic analytic manifold modulo $q-1$. A previous result by  \'A.M.\ Ledezma and the author on hearing Serre invariants then yields the asserted congruence.
\end{abstract}



\section{Introduction}

According to a Theorem by Jean-Pierre Serre, a $p$-adic compact analytic manifold $X$ is ``just'' a disjoint union of copies of a $p$-adic ball, and any  number of $p$-adic balls disjointly covering $X$ is uniquely determined modulo $q-1$, where $q$ is the cardinality of the residue field of the given non-archimedean local field $K$, over which $X$ is defined as analytic manifold \cite{Serre1965}. This gives one the first impression that compact $p$-adic analytic manifolds are uninteresting apart from that quantity, called \emph{Serre invariant} by \'A.M.\ Ledezma and the author. The reality is quite contrary, as most $p$-adic analytic manifolds studied by many mathematicians are much richer in structure. 
Such structure is already encoded in an integral structure, as explained in \cite{BKL2026}. The ongoing project \cite{brad_habil} is to construct $p$-adic Laplacian integral operators using such structure in order to exhibit number-theoretic properties from $p$-adic analytic manifolds underlying the $K$-rational points of algebraic varieties via boundary value problems, among other ways. 
\newline

Previous work on Laplacian diffusion on $p$-adic domains is meanwhile quite abundant, and it is the work by W.A.\ Z\'u\~{n}iga-Galindo on $p$-adic diffusion on networks \cite{Zuniga2020}, which eventually gave the impetus to the  author's recent research on diffusion on $p$-adic analytic manifolds using  differential forms \cite{Brad_thetaDiff,Brad_heatMumfGenus}, and also actual atlantes with charts 
\cite{diffMfp,HearingSerre}. This allows to hear from the spectrum of $p$-adic Laplacian several kinds of invariants, like e.g.\ the genus, existence of $2$-torsion points, valuation parity of uniformising parameters, Serre invariants, or the number of points in the reduction of an elliptic curve.
Other lines of research (among others) on the mathematical side of $p$-adic analysis are e.g.\ scaling limit theorems \cite{PW2025,Weisbart2024} in stochastics. The dissertation \cite{AngelDiss} contains also a spectral theory of finite ultrametric spaces and applies it to the sciences. Such form a natural approximation of $p$-adic analytic processes in the case of a $p^f$-regularity of the ultrametric trees involved.  
\newline

Abelian varieties form well-studied examples of algebraic-geometry objects whose underlying compact $p$-adic manifold structure is helpful in better understanding their properties. The seminal article \cite{Oesterle1984}
studies their Tamagawa numbers also through a measure-theoretic lense. This measure is obtained via the \emph{canonical measure} on the integral points on the N\'eron model of an abelian variety. This construction of a so-called \emph{gauge measure} $\mu_\omega$ (in the language of A.\ Weil in \cite{WeilAAG}) uses a nowhere vanishing top differential form $\omega$ which exists due to compactness of the analytic manifold, cf.\ \cite{Serre1965}. It is the relationship between the Tamagawa factor $\Phi(\mathds{F}_q)$ w.r.t.\ the prime $p$ on the one hand, and the action of inertia of the separable absolute Galois group on the $\ell$-adic cohomology group $H^1_{\etale}(A_{K_s},\mathds{Q}_\ell)$ of  the base-extended abelian variety $A_{K_s}$ w.r.t.\ the separable closure $K_s$ of $K$ on the other hand, which yields the final result of this article. It uses the following integral Laplacian operator:
\[
\Delta^su(x)=\int_{X}k(x,y)(u(x)-u(y))\,d\mu_\omega(y)
\]
on the space $\mathcal{D}(X)$  of locally constant complex-valued functions on the compact $p$-adic analytic manifold $X$. The kernel function $k(x,y)$ is specified in Section 2 below and depends only on the geodetic distance, a concept developed in \cite{diffMfp}, and also explained in that Section.
\newline

The main result of this article is stated here as:
\newline

\noindent
{\bf Corollary 3.5.}
{\em Let $A/K$ be an abelian variety with N\'eron model $\mathcal{A}$ over the integers $O_K$ of $K$. The wavelet eigenvalues $\lambda_\psi$ of the $p$-adic Laplacian operator $\Delta^s$  acting on the space $\mathcal{D}(A(K))$ (for wavelets with sufficiently small support) satisfy the congruence
\[
\lambda_\psi\equiv \absolute{\Phi(\mathds{F}_q)}\mod q-1
\]
for $s\in\mathds{R}$.
}
\newline

The following Section 2 reviews the method of hearing the Serre invariant of a compact $p$-adic analytic manifold, and Section 3 applies this method in order to hear the Tamagawa factor of an abelian variety defined over $K$.
\newline

In order to fix some notation, $K$ denotes a non-archimedean local field, and $O_K$ its ring of integers. Being a locally compact abelian group, $K^n$ has a Haar measure which is in general used in order to construct other measures via differential forms, here it is an $n$-form, where $n$ the dimension of the manifold, and usually written as $\omega$. The uniformising parameter of $K$ is denoted here as $\pi$.

\section{Hearing Serre invariants}

The notion of $p$-adic analytic manifold is treated in books like \cite{Igusa2001,Serre1992,Schneider2011,WeilAAG}. There is the following important result about the number $i(X)$ of copies of $O_K^n$, into which a compact $p$-adic analytic $n$-manifold $X$ decomposes:

\begin{thm}[Serre]\label{Serre}
Let $X$ be a compact $p$-adic analytic $n$-manifold. Then
\begin{enumerate}
\item There exists a nowhere vanishing analytic differential $n$-form on $X$.
\item If $\omega$ is a nowhere vanishing analytic differential $n$-form on $X$, then
\[
i(X) \equiv a \equiv
\absolute{\omega(x)} \mod q - 1\,,
\]
where the identity
\[
\int_X \absolute{\omega} =\frac{a}{q^b}
\]
holds true for some $a, b\in\mathds{N}$.
\end{enumerate}
\end{thm}

\begin{proof}
\cite[Th\'eor\`eme (2)]{Serre1965}.
\end{proof} 

The number $i(X)$ is what is called by
us the \emph{Serre invariant} of $X$.
\newline

Serre's result makes compact $p$-adic analytic manifolds seemingly uninteresting. However, through additional structure they do become interesting objects to study. Integral structures are such an example. For $R=O_K$, they are called $R$-structure in \cite{BKL2026}. On the vector space
$K^n$, an integral structure is the $O_K$-submodule spanned by a $K$-basis of $K^n$. This notion extends to the tangent spaces of a $p$-adic analytic manifold $X$ 
for a basis $e_1,\dots,e_n$ of $K^n\cong T_xX$ at each point $x\in X$ by  the assignment
\[
x\mapsto\Lambda_x=O_Ke_1\oplus\dots\oplus O_Ke_n\,,
\]
under the condition that it is locally constant, and this is what makes it into an integral structure $\Lambda$ on the manifold $X$. Cf.\ \cite{BKL2026} for more details. 
Taking the operation
\[
K^n\to\bigwedge^n K^n\cong K,\; e_1,\dots,e_n\mapsto e_1\wedge\dots\wedge e_n\,,
\] 
leads to a differential $n$-form on $X$ via the usual way of defining a linear form from a given vector on each $T_xX$.

\begin{definition}
An atlas $\mathcal{A}$ on a $p$-adic analytic manifold is said to be $O_K$-compatible, if for all $(U,\phi),(V,\psi)\in\mathcal{A}$ and $x\in U\cap V$ it holds true that the
derivative 
$T_x\tau_{UV} = \tau'_{UV}$ of the transition function
\[
\tau_{UV}\colon\phi(U)\to\psi(V)
\]
satisfies $T_x\tau_{UV}\in\GL_n(O_K)$.
\end{definition}

\begin{Lemma}
Every integral structure on $X$ arises from an $O_K$-compatible
atlas, and vice versa.
\end{Lemma}

\begin{proof}
Cf.\ \cite[Lemma 3.2.6]{BKL2026} and before.
\end{proof}

Given a compact $p$-adic analytic manifold $X$ with integral structure $\Lambda$, we fix a finite $O_K$-compatible atlas $\mathcal{A}$ giving rise to $\Lambda$. Under the mild condition that each chart $(U,\phi)$ of $\mathcal{A}$ has a distinct open $U\subseteq X$, there is a simplicial complex $N(\mathcal{A})$ whose faces are given by the non-empty (finite) intersections of these sets $U$ which form an open covering of $X$. It is called the \emph{nerve complex} of $X$ associated with $\mathcal{A}$.
Assume that the nerve complex is a connected simplicial complex.
\newline

In \cite{diffMfp}, a connected nerve complex plus integral  structure $\Lambda$ have been used in order to define a \emph{geodetic distance} on $X$, denoted as $d_\Lambda(x,y)$. It is obtained through extending the nerve complex by attaching trees of $p$-adic balls of the same dimension as $X$ to each point in the face poset of $N(\mathcal{A})$. In order to understand this, first notice that since the transition maps are $O_K$-bi-analytic, they take $p$-adic balls in $K^n$ to $p$-adic balls of the same radius. In this way, the notion of \emph{ball} on $X$ becomes well-defined. Since $X$ is a disjoint union of balls (now in this sense), the tree attached to a set $U_\sigma$ belonging to a face $\sigma$ in the face poset of $N(\mathcal{A})$ is simply given by the union of maximal balls of $X$ contained in $U_\sigma$, and their now regular decomposition into $p$-adic balls (locally in a chart containing $U_\sigma$). At the boundary of this infinite partially ordered set are the ($K$-rational) points of $X$. The distance $d_\Lambda(x,y)$ is now given via the geodetic path between $x$ and $y$, and its length is the sum of the measures of the disjoint open sets associated with the vertices along the path. More details can be found in \cite{diffMfp}. The only thing required here is how the measure used is related to the integral structure $\Lambda$: namely, the wedge $e_1(x)\wedge\dots\wedge e_n(x)$ on each $\wedge^nT_xX\cong K$ defines via duality a nowhere vanishing and everywhere defined differential $n$-form $\omega$ on $X$, and this in turn, in the usual manner as explained e.g.\ in \cite[Chapters 7.4, 7.5]{Igusa2001}, defines a measure $\mu_\omega$ on the $p$-adic analytic manifold $X$. In \cite{HearingSerre} find a brief exposition of that procedure.
\newline

Having at our disposal a measure $\mu_\omega$ and a geodetic distance
$d_\Lambda$ on a compact $p$-adic analytic manifold $X$, it is now possible to define integral Laplacian operators on the space $\mathcal{D}(X)$ of 
locally constant complex-valued functions on $X$ as follows:
\begin{align}\label{Laplacian}
\Delta_w u(x)
=\int_Xw(d_\Lambda(x,y))(u(x)-u(y))\,d\mu_\omega(y)
\end{align}
for $u\in\mathcal{D}(X)$, and where
$w(d_\Lambda(x,y))$ is a kernel function depending on the distance $d_\Lambda(x,y)$.
In \cite{diffMfp}, the kernel function
\[
w(\xi)=\xi^{-s}
\]
for $\xi\in\mathds{R}_+$ and $s\in\mathds{R}$, is used, wheras it is
\begin{align}\label{SerreKernel}
w(d_\Lambda(x,y))=\begin{cases}
d_\Lambda(x,y)^{-s}=\mu_\omega(x\wedge y)^{-s},&\text{$x\wedge y$ exists}
\\
1,&\text{otherwise}\
\end{cases}
\end{align}
in \cite{HearingSerre}.
The expression $x\wedge y$ stands for the supremum in the poset of faces and balls in the extended nerve complex, where we emphasise the convention that a face of $N(\mathcal{A})$ is never considered to be a ball, even if it may be the case that it is technically ball-shaped. This is necessary, in order to have avoid confusion in general statements about balls, e.g.\ their supremum being unique (if it exists): it is either a ball or a face. Notice that suprema in the poset associated with the finite simplicial complex $N(\mathcal{A})$ need not be 
unique in general.
\newline

The  operator used in this article is $\Delta_w$ from (\ref{Laplacian}) with $w$ defined as in (\ref{SerreKernel}), and $s\in\mathds{R}$.
The result about wavelet eigenvalues for this operator, denoted henceforth as $\Delta^s$, is needed later on in this article:

\begin{Proposition}
The wavelets $\psi$ of $X$ with sufficiently small support are eigenfunctions of $\Delta^s$ with eigenvalue
\[
\lambda_\psi=\int_{X\setminus U(x)}\,d\mu_\omega
+
\int_{U(x)\setminus B}\mu_\omega(x \wedge y)^{-s}\,d\mu_\omega(y) 
+ \mu_\omega(B)^{n-s/n}
\]
with $x\in\supp(\psi)=B\subset U(x)\subseteq X$, where $U(x)$ is the largest ball in $X$ containing
$x \in X$, and with $s\in\mathds{R}$. The eigenvalue does not depend on $x\in B$.
\end{Proposition}

\begin{proof}
\cite[Proposition 3.3]{HearingSerre}.
\end{proof}

In order to create certainty about terminology: first, a \emph{Kozurev wavelet} on $K^n$ is a function
\[
\psi_{B(a),j}(x)
=\mu(B(a))^{-\frac12}\chi(\pi^{d-1}\eta(j)x)1_{B(a)}(x)\,,
\]
where $\mu$ is the Haar measure on $K^n$ s.t.\ $\mu(O_K^n)=1$, $a\in K^n$, $B(a)\subset K^n$ a ball of radius $q^{-d}$ centred in $a$,
$j\in\left[\left(O_K/\pi O_K\right)^\times\right]^n$\,,
\[
\eta\colon(O_K/\pi O_K)^n\to K^n
\]
a lift of the canonical projection 
\[
O_K^n\to(O_K/\pi O_K)^n
\]
with $\pi\in O_K$ a uniformiser of $K$,
and $\chi\colon K^n\to S^1$ a unitary additive character. A Kozyrev wavelet is locally constant with compact support. Since transition functions are $O_K$-analytic, the radius of a ball in $X$ is well-defined, and thus a wavelet on $X$ is defined as being a Kozyrev wavelet on any chart containing its support, which is a ball centred in $a\in X$ and with fixed radius. The following theorem links wavelet eigenvalues to the Serre invariant: 

\begin{thm}\label{hearSerreMfp}
 Let $X$ be a compact $p$-adic analytic manifold, $s \in\mathds{R}$, and let $\psi$ be a
 wavelet in $X$ with sufficiently small support. Then the corresponding eigenvalue $\lambda_\psi$ of $\Delta^s$ satisfies
\[
\lambda_\psi \equiv i(X)
\mod q - 1\,,
\]
i.e. the Serre invariant $i(X)$ of   $(X, \omega)$ can be read off the wavelet eigenvalues of $\Delta^s$.
\end{thm}

\begin{proof}
\cite[Theorem 4.2]{HearingSerre}
\end{proof}


\section{Abelian varieties}

The following is often referred to as \emph{Lang's Theorem}:

\begin{thm}[Lang]\label{Lang}
Let $G$ be an algebraic group over a finite field with $q$ elements.
The map
\[
G\mapsto G,\;x\mapsto x^{q-1}
\]
is surjective.
\end{thm}

\begin{proof}
\cite[Corollary to Theorem 1]{Lang1956}.
\end{proof}

It will be used in what follows.
\newline 

Let $\mathcal{A}$ be the N\'eron model of an abelian variety $A$ of dimension $n$ over $K$, assumed to have  characteristic $p>0$.
By the N\'eron 
property of the $O_K$-model $\mathcal{A}$ of $A$, it follows that
\[
A(K)=\mathcal{A}(O_K)\,,
\]
and the \emph{canonical measure} $\mu_\omega$ on this $p$-adic analytic manifold  is given by a global section
\[
\omega\in\Gamma(\mathcal{A},\Omega^{n}_{\mathcal{A}/O_K})\,,
\]
and this is an invariant nowhere vanishing differential form. It arises naturally from the natural $O_K$-structure on $\mathcal{A}(O_K)$, cf.\ \cite[Proposition 3.2.15]{BKL2026}.
Denote the  identity component of $\mathcal{A}\to S$ with $S=\Spec(O_K)$ as $\mathcal{A}^\circ$. It sits inside
the exact sequence
\begin{align}\label{identityCompSeq}
0\to\mathcal{A}^\circ\to\mathcal{A}\to\Phi\to 0
\end{align}
in which $\Phi$ is the component group of $\mathcal{A}$. Cf.\ \cite[Chapter 7]{BLR1990} for more details.
\newline

The local $L$-function attached to the abelian variety $A$
is
\begin{align}\label{localL_AV}
L(A,t)=\det\left(1-F^{-1}t\mid H_{\etale}^1(A_{K_s},\mathds{Q}_\ell)^I\right)\in\mathds{Q}[t]\,,
\end{align}
where
$F\in\Gal(K_s/K)$ the Frobenius automorphism acting on 
the first $\ell$-adic cohomology group $H_{\etale}^1(A_{K_s},\mathds{Q}_\ell)$ 
with values in the constant sheaf $\mathds{Q}_\ell$ with $\ell$ a prime number distinct from $p$, and
\[
A_{K_s}=A\times_K K_s
\]
the base extension w.r.t.\ the separable closure $K_s$ of $K$.
Finally, $I\subset\Gal(K_s/K)$ is the inertia subgroup, defined as the kernel of the natural homomorphism 
\[
\rho\colon\Gal(K_s/K)\to\Gal(k_s/k)\,,
\]
where $k$ is the residue field of $K$, and $k_s$ that of $K_s$,
cf.\ \cite[Chapitre IV, \S1]{Serre1968}.
The set of fixed points under the action of $I=\ker(\rho)$ on $H_{\etale}^1(A_{K_s},\mathds{Q}_\ell)$ is denoted as:
\[
H_{\etale}(A_{K_s},\mathds{Q}_\ell)^I=\mathset{g\in H_{\etale}^1(A_{K_s},\mathds{Q}_\ell)\mid\forall \sigma\in I\colon g^\sigma=g}\,.
\]
The action of the inertia $I$ on the $\ell$-adic cohomology groups can be described along the lines of 
\cite[Section 2]{SZ2001} as follows: the Tate module 
\[
T_\ell(A)=\lim\limits_{\longleftarrow}{A}[\ell^m]
\]
appears in the $k$-th $\ell$-adic  cohomology group as follows:
\[
H_{\etale}^k(A_{K_s},\mathds{Z}_\ell)
=\Hom\left(\wedge^k\,T_\ell(A),\mathds{Z}_\ell\right)\,,
\]
which is a free $\mathds{Z}_\ell$-module of rank ${2n \choose k}$.
Since the Tate module for $\ell\neq p$ is a free $\mathds{Z}_p$-module, we thus have
\[
H_{\etale}^k(A_{K_s},\mathds{Q}_\ell)=H_{\etale}^k(A_{K_s},\mathds{Z}_\ell)\otimes_{\mathds{Z}_\ell}\mathds{Q}_\ell
\]
for $k\in\mathds{N}$. This means that the action on $\ell$-adic cohomology is defined entirely by its action on $H^1_{\etale}(A_{K_s},\mathds{Q}_\ell)$, which comes naturally from
the action of $\Gal(K_s/K)$ on the Tate module $T_\ell(A)$. This is an $\ell$-adic representation
\[
\Gal(K_s/K)\to\GL(T_\ell(A))\cong\GL_{2n}(\mathds{Z}_\ell)
\]
and induces the action of $I$ on $H^1_{\etale}(A_{K_s},\mathds{Q}_\ell)$. Now back to the sequence (\ref{identityCompSeq}), evaluated at $\mathds{F}_q$-rational points. From Theorem \ref{Lang} (Lang's Theorem), it follows that the sequence
\begin{align}\label{identityCompSeq_p}
\xymatrix{
0\ar[r]&\mathcal{A}^\circ(\mathds{F}_q)\ar[r]&\mathcal{A}(\mathds{F}_q)\ar[r]&\Phi(\mathds{F}_q)\ar[r]&0
}
\end{align}
is exact, and the quantity $\absolute{\Phi(\mathds{F}_q)}\in\mathds{N}$ is called the \emph{Tamagawa factor} w.r.t.\ the prime $p$.

\begin{Lemma}\label{Tamagawa}
It holds true that
$\mu_\omega(A(K))=\absolute{\Phi(\mathds{F}_q)}\cdot L(A,q^{-1})$.
\end{Lemma}

\begin{proof}
This follows from the following identities:
\begin{align*}
\mu_\omega(A(K))&=\absolute{\Phi(\mathds{F}_q)}\cdot\mu_\omega(\mathcal{A}^\circ(O_K))&(I)
\\
\mu_\omega(\mathcal{A}^\circ(O_K))&=q^{-\dim(A)}\absolute{\mathcal{A}^\circ(\mathds{F}_q)}&(II)
\\
\absolute{\mathcal{A}^\circ(\mathds{F}_q)}&=q^{\dim(A)}\cdot L\!\left(A,q^{-1}\right)&(III)
\end{align*}
The first identity $(I)$ follows from the exactness of the sequence (\ref{identityCompSeq_p}).

\smallskip
Identity $(II)$ is Andr\'e Weil's result on counting points with the canonical measure \cite[Theorem 2.2.5]{WeilAAG}.

\smallskip
Identity $(III)$ is shown in the Lecture Notes \cite[Theorem 4.1]{Conrad2015}.
\end{proof}

\begin{Lemma}\label{LocalLAV}
It holds true that
$L(A,q^{-1})\equiv 1\mod q-1$.
\end{Lemma}

\begin{proof}
Since $q\equiv 1\mod q-1$, the number $L(A,q^{-1})$ is congruent with the constant term of the polynomial $L(A,t)\in\mathds{Q}[t]$ modulo $q-1$. Since the latter equals $1$, the assertion follows.
\end{proof}

\begin{cor}\label{SerreTamagawa}
It holds true that $i(A(K))\equiv\absolute{\Phi(\mathds{F}_q)}\mod q-1$.
\end{cor}

\begin{proof}
This is an immediate consequence of Lemmas \ref{Tamagawa} and \ref{LocalLAV}, together with 
Serre's result (Theorem \ref{Serre}).
\end{proof}

The upshot is that Tamagawa factors can be heard modulo $q-1$ via $\Delta^s$.

\begin{cor}
Let $A/K$ be an abelian variety with N\'eron model $\mathcal{A}/O_K$. 
The wavelet eigenvalues $\lambda_\psi$ of the operator $\Delta^s$ acting on $\mathcal{D}(A(K))$ (for wavelets with sufficiently small support) satisfy
\[
\lambda_\psi\equiv\absolute{\Phi(\mathds{F}_q)}\mod q-1
\]
for $s\in\mathds{R}$.
\end{cor}

\begin{proof}
This now follows immediately from
Theorem \ref{hearSerreMfp} and Corollary \ref{SerreTamagawa}.
\end{proof}

\section*{Acknowledgements}


Frank Herrlich, Stefan K\"uhnlein, \'Angel Mor\'an Ledezma, 
David Weisbart, Evgeny Zelenov, and Wilson Z\'u\~{n}iga-Galindo are thanked for valuable discussions. 

\bibliographystyle{plain}
\bibliography{biblio}

\end{document}